\documentclass{amsart}
\usepackage{amsthm,amssymb,color}
\usepackage{hyperref}
\hypersetup{
  colorlinks=true
}

\newcommand\bB{{\mathbf B}}
\newcommand\bG{{\mathbf G}}
\newcommand\bH{{\mathbf H}}
\newcommand\bI{{\mathbf I}}
\newcommand\bS{{\mathbf S}}
\newcommand\bT{{\mathbf T}}
\newcommand\bU{{\mathbf U}}
\newcommand\bW{{\mathbf W}}
\newcommand\ba{{\mathbf a}}
\newcommand\bb{{\mathbf b}}
\newcommand\bc{{\mathbf c}}
\newcommand\bi{{\mathbf i}}
\newcommand\bj{{\mathbf j}}
\newcommand\bn{{\mathbf n}}
\newcommand\bp{{\mathbf p}}
\newcommand\bs{{\mathbf s}}
\newcommand\bt{{\mathbf t}}
\newcommand\bv{{\mathbf v}}
\newcommand\bw{{\mathbf w}}
\newcommand\BF{{\mathbb F}}
\newcommand\BQ{{\mathbb Q}}
\newcommand\BZ{{\mathbb Z}}
\newcommand\CA{{\mathcal A}}
\newcommand\CB{{\mathcal B}}
\newcommand\CC{{\mathcal C}}
\newcommand\etc{etc\dots}
\newcommand\fS{{\mathfrak S}}
\DeclareMathOperator\ch{\mathrm{char}}
\DeclareMathOperator\Id{\mathrm{Id}}
\DeclareMathOperator\Sl{\mathrm{SL}}
\DeclareMathOperator\GL{\mathrm{GL}}
\DeclareMathOperator\Sp{\mathrm{Sp}}
\DeclareMathOperator\Spin{\mathrm{Spin}}
\DeclareMathOperator\diag{\mathrm{diag}}
\newcommand\Gsc{{\bG_{\mathrm{sc}}}}
\newcommand\Tsc{\bT_{\mathrm{sc}}}
\newcommand\Tad{\bT_{\mathrm{ad}}}
\newcommand\Fq{\BF_q}
\newcommand\lexp[2]{\kern\scriptspace\vphantom{#2}^{#1}\kern-\scriptspace#2}
\newcommand\ts{{\boldsymbol\sigma}}
\newcommand\tpi{{\tilde{\boldsymbol\pi}}}
\newcommand\bpi{{\boldsymbol\pi}}
\newcommand\inv{^{-1}}
\newtheorem{proposition}[equation]{Proposition}
\newtheorem{corollary}[equation]{Corollary}
\newtheorem{lemma}[equation]{Lemma}
\newtheorem{definition}[equation]{Definition}
\newtheorem{theorem}[equation]{Theorem}
\numberwithin{equation}{section}
\theoremstyle{remark}
\newtheorem{remark}[equation]{Remark}
\newcommand\Chevie{\texttt{Chevie}}

\subjclass[2020]{ 20G15, 20G40}

\title[Centralisers of semisimple elements]{Centralisers of semisimple elements are
semidirect products}
\author{Fran\c cois Digne and Jean Michel}
\address[F.~Digne]{LAMFA, CNRS UMR 7352, Universit\'e de Picardie-Jules Verne,
F-80039 Amiens, France.}
\email{digne@u-picardie.fr}
\urladdr{www.lamfa.u-picardie.fr/digne}
\address[J.~Michel]{Universit\'e Paris Cit\'e, Sorbonne Universit\'e,
CNRS, IMJ-PRG, F-75013 Paris, France.}
\email{jean.michel@imj-prg.fr}
\urladdr{webusers.imj-prg.fr/$\sim$jean.michel}
\begin{document}
\begin{abstract}
Let $\bG$ be a connected reductive algebraic group over an algebraically closed
field,  and  let  $s\in\bG$  be  a  semisimple  element.  We  show that the
centraliser of $s$ is the semidirect product of its identity component by
its group of components.

We  then look at the case where  $\bG$ is defined over an algebraic closure
of a finite field $\Fq$, and $F$ is an endomorphism such that some power
is a Frobenius endomorphism attached to an $\Fq$-structure  on $\bG$. 
We show that if the centraliser of $s$ is $F$-stable
we have a semidirect product decomposition of its $F$-fixed points.
\end{abstract}
\maketitle
\section{Introduction}
Let $\bG$ be a connected reductive algebraic group over an algebraically closed field $k$, 
and let $s\in\bG$ be a semisimple element. We define $A_\bG(s)$ by the
exact sequence $$1\to C_\bG(s)^0\to C_\bG(s)\to A_\bG(s)\to 1\eqno(*)$$ 
In the first section we show
\begin{theorem}\label{principal}
The sequence $(*)$ always splits, that is
there is a finite subgroup $A_0\subset C_\bG(s)$ such that $C_\bG(s)$
is a semidirect product $C_\bG(s)^0\rtimes A_0$.
\end{theorem}

In the second section we assume that $\bG$ is defined over an algebraic
closure of a finite field $\Fq$ and we look at the fixed points $\bG^F$ of
$\bG$ under an endomorphism $F$ such that some power is a Frobenius endomorphism
attached to an $\Fq$-structure on $\bG$. We assume that $C_\bG(s)$ is $F$-stable, so
we have the exact sequence 
$$1\to C_\bG(s)^{0F}\to C_\bG(s)^F\to A_\bG(s)^F\to 1\eqno(\sharp)$$ 
We show 
\begin{theorem}\label{principal2}
The sequence $(\sharp)$ splits, 
that is  there is  a finite  subgroup $A_0\subset
C_\bG(s)^F$ such that $C_\bG(s)^F$ is a semidirect product
$C_\bG(s)^{0F}\rtimes A_0$.
\end{theorem}
These results have already been obtained for adjoint groups of type $D$
in \cite{cabanes-spaeth}. We use different methods.

We  would like to thank Michel Brou\'e for making us aware of the questions
that  we  answer  in  Theorems  \ref{principal}  and  \ref{principal2}; the
motivation  behind his query is that our positive answer shows that
semisimple  centralisers  are  more  ``generic''  than  might have been 
expected a priori.

We thank Gunter Malle and the referees for a thorough proofreading of previous
versions.
\section{Algebraic groups over an algebraically closed field}\label{section1}
Let  $\bT$ be a  maximal torus of $\bG$ such that $s\in\bT$. Let $\Phi$ be the
root  system of $\bG$ with respect to $\bT$,  and let $W=N_\bG(\bT)/\bT$
be the Weyl group of $\bG$ relative to $\bT$.
To $\alpha\in\Phi$ is associated a one-parameter subgroup
$\bU_\alpha$ normalised by $\bT$. Then we have $C_\bG(s)^0=\langle \bT,
\{\bU_\alpha\mid     \alpha\in\Phi(s)\}\rangle$,
a reductive group with root system
$\Phi(s)=\{  \alpha\in  \Phi\mid  \alpha(s)=1\}$
(see for example
\cite[Proposition 3.5.1]{DM}).     We    set
$W^0(s)=\langle  s_\alpha\in  W\mid  \alpha\in\Phi(s)\rangle$  and  $W(s)=\{w\in
W\mid\lexp  ws=s\}$. The group $C_\bG(s)$ is generated by $C_\bG(s)^0$ and
$C_{N_\bG(\bT)}(s)$, the latter being the preimage of $W(s)$ in $N_\bG(\bT)$
(\cite[Proposition 3.5.1]{DM}).
Then  we  have $A_\bG(s)\simeq W(s)/W^0(s)$ (see for example
\cite[Remark 3.5.2]{DM}).
Splitting  the  sequence  $(*)$  is  a consequence of  splitting the exact
sequence
$$1 \to N_{C_\bG(s)^0}(\bT) \to  N_{C_\bG(s)}(\bT)\to A_\bG(s)\to1\eqno{(**)}$$
since if $N_{C_\bG(s)}(\bT)=N_{C_\bG(s)^0}(\bT)\rtimes A_0$ for some group
$A_0$ then $C_\bG(s)=C_\bG(s)^0\rtimes A_0$.

Since $W^0(s)$ acts simply transitively on the bases of $\Phi(s)$, 
for any choice of a system $\Phi^+(s)$ of positive roots in $\Phi(s)$,
we have the following
\begin{proposition}\label{AW(s)} The sequence  $1\to W^0(s)\to W(s)\to A_\bG(s)\to 1$
splits, by lifting $A_\bG(s)$ to 
$A_W(s):=\{w\in W(s)\mid w(\Phi^+(s))=\Phi^+(s)\}$.
\end{proposition}
For any choice of $\Phi^+(s)$ we have a group $A_W(s)$.
We will find a choice $\Phi^+(s)$ such that we can lift the corresponding
subgroup $A_W(s)$ by a group morphism to $N_{C_\bG(s)}(\bT)$, 
splitting thus $(**)$.
Since the various choices for $\Phi^+(s)$ are conjugate under $W^0(s)$, by
lifting the conjugating element to $N_{C_\bG(s)}(\bT)$ we can lift
$A_W(s)$ for any other choice.
\subsection{Reduction to the case $s$ of finite order}
To prove that $(*)$ splits, we may assume that $s$ has finite order
by \cite[Proposition 1.5]{bonnafe},
which  states that $s$  has same centraliser  as some semisimple element of
finite order.

\subsection{The Tits lifting, and reduction to $\ch k\ne 2$}\label{Tits}
Assume that $\bG$ is defined over an algebraically closed field $k$ of
characteristic $p$.
Choosing an  isomorphism between  $(\BQ/\BZ)_{p'}$ (the elements of order prime to $p$ of
$\BQ/\BZ$, where by convention
$(\BQ/\BZ)_{p'}=\BQ/\BZ$ when $p=0$) and the
roots  of unity  in $k^\times$,  we can  identify the  elements of $\bT$ of
finite  order with $Y(\bT)\otimes(\BQ/\BZ)_{p'}$ where $Y(\bT)$ is the group
of cocharacters of
$\bT$.

We denote by $\Phi^\vee\subset Y(\bT)$ the coroots of $\bG$ with respect to 
$\bT$ and by $\alpha\mapsto\alpha^\vee$ the natural
bijection $\Phi\xrightarrow\sim\Phi^\vee$.

We  choose a Borel subgroup
$\bG\supset\bB\supset\bT$,  which defines a basis $\Pi$  of $\Phi$
and a system $\Phi^+$  of positive roots; it
also defines a Coxeter system $(W,S)$ where $S=\{s_\alpha\mid\alpha\in\Pi\}$.

We will use a Tits lifting (see \cite{Tits}), a lifting 
$\sigma: W\to N_\bG(\bT)$ which is deduced from a lifting of $S$
satisfying the relations
\begin{itemize}
 \item[(i)] If $m$ is the order of $st$ for $s,t\in S$ then
$\underbrace{\sigma(s)\sigma(t)\sigma(s)\cdots}_{m\textrm{ terms}}=
\underbrace{\sigma(t)\sigma(s)\sigma(t)\cdots}_{m\textrm{ terms}}$.
\item[(ii)] Denoting by $s\mapsto\alpha_s$ the bijection $S\simeq\Pi$,  for $s\in S$ we
have $\sigma(s)^2=\alpha_s^\vee/2\in Y(\bT)\otimes(\BQ/\BZ)_{p'}$ if $p\neq 2$ and
$\sigma(s)^2=1$ if $p=2$.
\end{itemize}
It follows from (i) (see \cite[Chapter IV, no 1.5, Proposition 5]{bbk})
that for $w\in W$ we can define
$\sigma(w)$ by $\sigma(s_1)\ldots\sigma(s_r)$ if $w=s_1\ldots s_r$ is any
reduced expression, thus deducing a lifting of $W$. 

Let $B(W)=\langle\bs\in\bS\mid
\underbrace{\bs\bt\bs\cdots}_{m\textrm{ terms}}=
\underbrace{\bt\bs\bt\cdots}_{m\textrm{ terms}}\rangle$ be the braid group of
$W$ where $\bS$ is a copy of $S$ and $m$ is the order of $st$.
It follows also from (i) there is a well-defined group morphism
$B(W)\xrightarrow\ts N_\bG(\bT)$ defined by $\ts(\bs)=\sigma(s)$. If we
define the lift $w\mapsto\bw: W\to\bW\subset B(W)$ by lifting a reduced
expression of $w$, we have more generally $\ts(\bw)=\sigma(w)$.

If $p=2$, it follows from $\sigma(s)^2=1$ for $s\in S$ that $\sigma(W)$ is a 
group isomorphic to $W$, thus $(*)$
splits in this case by lifting $A_W(s)$ to $\sigma(A_W(s))$. 
We assume from now on in this section that $p\ne 2$.

\subsection{Reduction to the case of semisimple groups}

\begin{proposition}\label{derived}
Let $\bG'$ be the derived subgroup of $\bG$ and $\bT'=\bT\cap\bG'$. If $(**)$ splits in 
$\bG'$ for any $s'\in\bT'$ then it splits in $\bG$ for any $s\in\bT$.
\end{proposition}
\begin{proof} We have $\bG=(Z\bG)^0\cdot \bG'$. Let $s$ be a semisimple element of $\bT$.
Write  $s=s'z$ with  $s'\in\bT'$ and  $z\in(Z\bG)^0$. 
By assumption $N_{C_{\bG'}(s')}(\bT')=N_{C_{\bG'}(s')^0}(\bT')\rtimes  A_0$ for some
subgroup $A_0$ of $N_{C_{\bG'}(s')}(\bT')$. We have
$N_{C_\bG(s)}(\bT)=(Z\bG)^0 N_{C_{\bG'}(s')}(\bT')=
(Z\bG)^0.N_{C_{\bG'}(s')^0}(\bT')A_0 =N_{C_\bG(s)^0}(\bT).A_0$.
To prove that $N_{C_\bG(s)}(\bT)=N_{C_\bG(s)^0}(\bT)\rtimes A_0$,
it remains to prove that $A_0\cap N_{C_\bG(s)^0}(\bT)=\{1\}$. If $a_0$ is in $A_0\cap
N_{C_\bG(s)^0}(\bT)$, it can be
written $a_0=z_0 n'$ with $n'\in N_{C_{\bG'}(s')^0}(\bT')$ and $z_0\in (Z\bG)^0$. Then
$z_0\in(Z\bG)^0\cap\bG'\subset Z\bG'\subset\bT'\subset N_{C_{\bG'}(s')^0}(\bT')$,
so that $a_0$ is in $N_{C_{\bG'}(s')^0}(\bT')\cap A_0=\{1\}$.
\end{proof}

\subsection{Proof of Theorem \ref{principal}}

According  to  the  previous  reductions,  we  assume  now  that  $\bG$  is
semisimple and $\ch k\ne 2$.

We recall that we have chosen in Subsection \ref{Tits} a basis $\Pi$ of the
root system $\Phi$ of $\bG$ with respect to $\bT$.
\begin{definition} \label{decomposition} Let $\Phi=\Phi_1\cup\ldots\cup\Phi_n$ be the decomposition
into irreducible subsystems of the root system of $\bG$, and let 
$\tilde\alpha_i$  be the negative of the highest root of $\Phi_i$ (for the
order defined by $\Pi\cap\Phi_i$). We define
$\CA=N_W(\Pi\cup\{\tilde\alpha_1,\ldots,\tilde\alpha_n\})$. 
\end{definition}

There exists an exact sequence
$$1\to Z\to\Gsc\xrightarrow\pi\bG\to 1$$
where $\Gsc$ is semisimple simply connected and $Z\subset Z\Gsc$.

Let $\Tad=\bT/Z\bG=\Tsc/Z\Gsc$ where $\Tsc=\pi\inv(\bT)$; by
\cite[(3.7)]{bonnafe} we have an isomorphism
$\varpi^\vee:\CA\xrightarrow\sim Y(\Tad)/Y(\Tsc)$. Using the isomorphism
\cite[Lemma 3.1]{bonnafe} $(Y(\Tad)/Y(\Tsc))_{p'}\xrightarrow\sim Z\Gsc$
we deduce an isomorphism $\iota:\CA_{p'}\xrightarrow\sim Z\Gsc$. 
We define $\CA_\bG\subset\CA$ as $\varpi^{\vee-1}(Y(\bT)/Y(\Tsc))$; 
then $\varpi^\vee$ restricts to an isomorphism 
$$(\CA_\bG)_{p'}\xrightarrow\sim(Y(\bT)/Y(\Tsc))_{p'}=:\Pi_1(\bG),$$
the fundamental group of $\bG$, and by \cite[Lemma 3.1]{bonnafe}
$\iota$ restricts
to an isomorphism $\iota:(\CA_\bG)_{p'}\xrightarrow\sim Z$.

The element $s$ is identified with an
element of $Y(\bT)\otimes(\BQ/\BZ)_{p'}$; let $\lambda$ be a representative
of $s$ in $Y(\bT)\otimes\BQ$. We have 
$W(s)=\{w\in W\mid w(\lambda)\equiv\lambda\pmod{Y(\bT)}\}$.
Since $Y(\bT)\rtimes W$ contains the affine Weyl group $Y(\Tsc)\rtimes W$,
we can replace $s$ by a $W$-conjugate and
$\lambda$  by  a  $Y(\bT)$-translate  in  order  to get
$\lambda\in\CC$, the closure of the fundamental alcove of the affine Weyl group defined as
$$\CC=\{\lambda\in  Y(\bT)\otimes\BQ\mid  \langle\alpha,\lambda\rangle\ge  0
\text{ for }\alpha\in\Pi\text{ and }\langle\lambda,\tilde\alpha_i\rangle\ge
-1\text{ for }i\in 1,\ldots,n\}.$$

By \cite[3.14(a)]{bonnafe}, the set
$(\Pi\cap \Phi(s))\cup\{\tilde\alpha_i\mid
\langle\lambda,\tilde\alpha_i\rangle=-1\}$ is a basis of $\Phi(s)$.
We choose $\Phi^+(s)$  as the positive combinations of this basis.
The results of \cite{bonnafe} show that $A_W(s)$ as defined in Proposition \ref{AW(s)}
is a subgroup of $\CA_\bG$, precisely
we  have (see \cite[3.14(b)]{bonnafe})
$$\begin{aligned}
A_W(s)&=\{  c\in \CA_\bG\mid
c(\lambda)\equiv\lambda\pmod{Y(\bT)}\text{ and }
c(\Phi^+(s))=\Phi^+(s)\}\\
 &=\{c\in\CA_\bG\cap W(s) \mid c(\Phi^+(s))=\Phi^+(s)\}.
\end{aligned}
$$

Since $A_W(s)$ is a subgroup of $\CA_\bG$ the following proposition
proves Theorem \ref{principal}.
\begin{proposition}\label{group lifting}
We can lift $\CA_\bG$ by a group morphism to $N_\bG(\bT)$.
\end{proposition}
\begin{proof}
A part of the proof, that of Proposition \ref{lifting} in the case of simply 
connected quasisimple groups, will be done in the next subsection.

We start with a lemma showing what we can do for the whole of $\CA$: not a
group lifting, but something similar.
\begin{definition}\label{flat}
Let  $o(a)$ denote the  order of an  element $a$. We  say that a lift $\dot
a\in N_\Gsc(\Tsc)$ of $a\in\CA$ has property $(\flat)$ if
$${\dot a}^{o(a)}=\begin{cases}\iota(a^{o(a)/2})&\text{if $o(a)$ is even,}\\
\iota(a^{o(a)})=1&\text{otherwise.}\end{cases}\eqno{(\flat)}$$ 
\end{definition}
Note that $a^{o(a)/2}$ lies in $\CA_{p'}$ since $p\ne 2$, thus
$\iota(a^{o(a)/2})$ makes sense.

\begin{proposition} \label{lifting} 
There exists a lift $\tau$ of $\CA$ to a set of commuting
elements  of $N_\Gsc(\Tsc)$ such  that for each  $a\in\CA$ we have
$\tau(a)$ has property $(\flat)$.
\end{proposition}
\begin{proof}
We show here that, assuming Proposition \ref{lifting} holds for a
quasi-simple group (which we will show in the next subsection), 
it holds for $\Gsc$.

The  decomposition $\Phi=\Phi_1\cup\ldots\cup\Phi_n$ corresponds
to a direct product decomposition $\Gsc=\bG_1\times\cdots\times\bG_n$ where
the  $\bG_i$  are  quasi-simple  simply  connected  groups (see for example
\cite[Theorem 24.3]{milne}).

The group $\CA$ is a direct product
of  groups $\CA_i$ for each $\bG_i$, which  are cyclic unless $\bG_i$ is of
type  $D_{2k}$, in which case $\CA_i$ is isomorphic to $(\BZ/2\BZ)^2$. We
fix  a  generator  $c_i$  of  $\CA_i$  (resp.\  two generators $c_i,c'_i$ of
$\CA_i$  in case $\bG_i$ is  of type $D_{2k}$). 

We assume that Proposition \ref{lifting} holds for a
quasi-simple group. That is if $\bT_i=\Tsc\cap\bG_i$, there is a lift 
 $\tau(c_i)\in N_{\bG_i}(\bT_i)$ which has property $(\flat)$
(and the same for $c'_i$ when $\bG_i$ is of type $D_{2k}$;
further, in this case, $\tau(c_i)$ and $\tau(c'_i)$ commute).

\begin{lemma}\label{product} With the above notation, assume we have $\tau_i:\CA_i\to
N_{\bG_i}(\bT_i)$ such that $\tau_i(a)$ satisfies $(\flat)$ for all $a\in\CA_i$ and define 
$\tau$ for an arbitrary element $a=\prod_i a_i\in\Pi\CA_i$
by $\tau(a)=  \prod_i\tau_i(a_i)$; then
$\tau(a)$ satisfies $(\flat)$ for all $a\in\CA$.
\end{lemma}
\begin{proof} It is sufficient to prove the result for the direct product of two groups
$\bG_1\times\bG_2$ with $\CA=\CA_1\times \CA_2$. Let $a=(a_1,a_2)\in\CA$. The order of $a$
is the lcm $l$ of $o(a_1)$ and $o(a_2)$. If 
$o(a_1)$ and $o(a_2)$ are both odd then $o(a)$ is odd, and
$\tau(a)^{o(a)}=\tau_1(a_1)^{o(a)}\tau_2(a_2)^{o(a)}=1$;
if both orders are even,
we have $\tau(a)^l=\tau_1(a_1)^l\tau_2(a_2)^l=
\tau_1(a_1)^{o(a_1)\cdot(l/o(a_1))}\tau_2(a_2)^{o(a_2)\cdot(l/o(a_2))}
=\iota(a_1^{o(a_1)/2})^{l/o(a_1)}\iota(a_2^{o(a_2)/2})^{l/o(a_2)}=
\iota(a_1^{l/2})\iota(a_2^{l/2})=\iota(a^{l/2})$.
We assume now that $o(a_1)$ is even and $o(a_2)$ is odd.
We have $\tau(a)^l=\tau(a_1)^l\tau(a_2)^l=
\tau(a_1)^l=\tau(a_1)^{o(a_1)\cdot(l/o(a_1))}=\iota(a_1^{o(a_1)/2})^{l/o(a_1)}=
\iota(a_1^{l/2})=\iota(a_1^{l/2})\iota(a_2^{l/2})=\iota(a^{l/2})$, the second 
equality and the equality before last since $l/2$ is a multiple of $o(a_2)$.
\end{proof}
\end{proof}
We show now how to deduce Proposition \ref{group lifting} from Proposition \ref{lifting}.
This will complete the proof of Theorem \ref{principal}.
We define a lift $\tau_1$ of $\CA_\bG$ by $\pi(\tau(\CA_\bG))$ where $\tau$ is
as in Proposition \ref{lifting} and $\pi$ is the simply connected covering $\Gsc\to\bG$.
Condition $(\flat)$ implies in particular that if $a\in\CA_\bG$ then
$\tau(a)^{o(a)}\in Z$,
thus the image $\tau_1(\CA_\bG)$ in $N_\bG(\bT)$
is a set of commuting elements lifting $\CA_\bG$ such that for any
$a\in\CA_\bG$ we have that $o(a)$ is equal to the order of $\tau_1(a)$ 
(it cannot be less since $\tau_1(a)$ is a lift of $a$). Therefore the following lemma gives us
a lift of $\CA_\bG$ to $N_\bG(\bT)$ that is a group morphism. This concludes the proof of 
Proposition \ref{group lifting} and Theorem \ref{principal}, once Proposition \ref{lifting} is
proved.
\begin{lemma}\label{tau2}
If there exists a lift $\tau_1$ of $\CA_\bG$ to $N_\bG(\bT)$
consisting of commuting elements and
such that $o(\tau_1(a))=o(a)$ for all $a\in\CA_\bG$, then there exists another lift $\tau_2$ of
$\CA_\bG$ which is a group morphism $\CA_\bG\to N_\bG(\bT)$. 
\end{lemma}
\begin{proof}
We decompose $\CA_\bG$ into a product of cyclic subgroups
$\CA_\bG=\Pi_i\langle g_i\rangle$. We define $\tau_2(\Pi
g_i^{m_i})=\Pi_i(\tau_1(g_i))^{m_i}$; $\tau_2$ is well-defined since
$o(\tau_1(g_i))=o(g_i)$. Then
$\tau_2(\CA_\bG)$ is a group, the group generated by the $\tau_1(g_i)$,
and lifts $\CA_\bG$ by its definition.
\end{proof}
\end{proof}
We should explain  where the ``strange'' condition $(\flat)$ originates from.
Trying to find a lift $\tau_1$ satisfying the assumption of Lemma \ref{tau2}
for the group $\Sl_n/\mu_d$ where $\mu_d$ is the group of roots of unity of
order $d$, for some $d|n$, we found that starting from a lift to $\Sl_n$
satisfying $(\flat)$ would do the job (among other possibilities); the
advantage of condition $(\flat)$ is that it is preserved by a product
as shows Lemma \ref{product}.

\subsection{Proof of Proposition \ref{lifting}}
In this subsection we prove the part of Proposition \ref{lifting} which was deferred, that is
for the case of quasi-simple groups.
\begin{lemma} \label{lifting2} 
Assume that $\bG$ is a simply connected quasisimple group.
There exists a lift $\tau$ of $\CA$ to a set of commuting
 elements  of $N_\bG(\bT)$ such that for each  $a\in\CA$ the element
$\tau(a)$ has property $(\flat)$.
\end{lemma}

We will prove this case by case for each type.
For $I\subset S$ we let $w_I$ be the longest element of the Coxeter subgroup
$W_I$ generated by $I$ and set $w_0=w_S$.
We  let $c$ be  a generator of $\CA$ (resp.~$c,c'$ be generators in case
$D_{2k}$). The element $c$ is of the form 
$w_0w_{S-\{s_j\}}$ where $s_j$ corresponds to  some minuscule coweight
(a minuscule coweight is a coweight associated to a simple root $\alpha_j$ such that the coefficient
of the highest root of $\Phi$ on $\alpha_j$ is $1$, equivalently a
coweight which is a vertex of the fundamental alcove of the affine Weyl
group); similarly we have two minuscule coweights in type
$D_{2k}$ (see \cite[Section 3.B, Equation 3.5]{bonnafe}, where $w_0w_{S-\{s_j\}}$ is denoted
by $z_{\alpha_j}$).

We now generalise \cite[Lemma 5.4]{AV} to the braid group. We use the notation of Subsection 
\ref{Tits}.
We denote by $x\mapsto\tilde x$ the anti-automorphism of $B(W)$ extending
the identity on $\bS$. Thus if $\bw$ is the lift of $w\in W$ to
$\bW$, the lift of $w\inv$ is $\tilde\bw$.
\begin{proposition}\label{adams-vogan} Let
$\rho^\vee=(\sum_{\alpha\in\Phi^+}\alpha^\vee)/2$. Then for any
$\bb\in B(W)$ of image $w\in W$ we have $\ts(\bb)\ts(\tilde \bb)=
(\rho^\vee-w(\rho^\vee))/2\in Y(\bT)\otimes (\BQ/\BZ)_{p'}$.
\end{proposition}
\begin{proof}
 We proceed by induction on the terms of a decomposition of $\bb$ into a product
 of elements of
 $\bS^{\pm1}$. Writing $\bb=\bs\bb'$ where $\bs\in\bS^{\pm1}$ has image
 $s$ in $W$ and $\bb'$ has image $w'$ in $W$
 we get the sequence of equalities between elements of
 $Y(\bT)\otimes(\BQ/\BZ)_{p'}$:
 $$\begin{aligned}
  \ts(\bb)\ts(\tilde\bb)&=\ts(\bs)\ts(\bb')
  \ts(\tilde\bb')\ts(\bs)\inv\ts(\bs)^2=
  \ts(\bs)\ts(\bb')\ts(\tilde\bb')
  \ts(\bs)\inv+\alpha_s^\vee/2\\
  &=\ts(\bs)(\rho^\vee-w'(\rho^\vee))/2+\alpha_s^\vee/2\\
  &=s(\rho^\vee-w'(\rho^\vee))/2+\alpha_s^\vee/2\\
  &=(\rho^\vee-\alpha_s^\vee+\alpha_s^\vee-w(\rho^\vee))/2
 \end{aligned}$$
where the second equality uses the fact that
$\sigma(s)^2=\alpha_s^\vee/2 \in Y(\bT)\otimes(\BQ/\BZ)_{p'}$,
the third uses the induction hypothesis and the fifth uses
$s(\rho^\vee)=\rho^\vee-\alpha_s^\vee$
(see \cite[ Chap VI, Proposition 29(ii)]{bbk}).
\end{proof}
\begin{corollary}\label{sigma(involution)}
\begin{enumerate}
 \item For any $I\subset S$ we have $\sigma(w_I w_0)\sigma(w_0 w_I)=
  \rho^\vee-\rho^\vee_I\in Y(\bT)\otimes (\BQ/\BZ)_{p'}$ where 
 $\rho^\vee_I=(\sum_{\alpha\in\Phi_I^+}\alpha^\vee)/2$.
\item We have $\sigma(w_0)^2=\rho^\vee$ (this element is called the principal
involution of $\bG$).
\end{enumerate}
\end{corollary}
\begin{proof}
 The first formula is an immediate computation from Proposition
 \ref{adams-vogan}
and the obvious formula  $w_I(\rho^\vee)=\rho^\vee-2\rho^\vee_I$.
The second is the special case $I=\emptyset$.
\end{proof}
Note that the principal involution is a central element by the formula
$s_\alpha(\rho^\vee)-\rho^\vee=-\alpha^\vee\in Y(\bT)$, that is 
$s_\alpha(\rho^\vee)-\rho^\vee=0$ in $Y(\bT)\otimes (\BQ/\BZ)_{p'}$.

\begin{lemma}  \label{c^i} For $c\in\CA$, let $\tau(c)$ be a lift of $c$ to
$N_\bG(\bT)$  which satisfies $(\flat)$. Then, for  a power $c^i$, the lift
$\tau(c^i)=\tau(c)^i$ satisfies $(\flat)$ unless $o(c^i)$ is odd, $o(c)$ is
even   and  $i/\gcd(i,o(c))$  is  odd.  In   this  last  case  the  lift
$\tau(c^i)=\tau(c)^i\iota(c^{o(c)/2})$ satisfies $(\flat)$.
\end{lemma}
\begin{proof}
First note that since the order of $c^i$ is $o(c)/\gcd(i,o(c))$, we have
$io(c^i)=i_0o(c)$ where $i_0=i/\gcd(i,o(c))$.

If $o(c^i)$ is even then so is $o(c)$ and we have
$(\tau(c)^i)^{o(c^i)}=\tau(c)^{i_0o(c)}=
\iota(c^{o(c)/2})^{i_0}=\iota(c^{i_0o(c)/2})=\iota((c^i)^{o(c^i)/2})$.

If  $o(c^i)$  is  odd  we  must  find  a  lift $\tau'(c^i)$  such  that
$\tau'(c^i)^{o(c^i)}=1$. We have
$\tau(c^i)^{o(c^i)}=(\tau(c)^{o(c)})^{i_0}$.  This  last  term  is $1$ if
$i_0$ is even since $\tau(c)^{o(c)}$ is of order 1 or 2. It is also $1$ if
$o(c)$ is odd since then $\tau(c)^{o(c)}=1$. In these cases we can take
$\tau'(c^i)=\tau(c)^i$.

If $o(c^i)$ is odd, $i_0$ odd and $o(c)$ even let us check what happens
with $\tau'(c^i)=\tau(c)^i\iota(c^{o(c)/2})$. We have
$$(\tau(c)^i\iota(c^{o(c)/2}))^{o(c^i)}=\tau(c)^{i_0o(c)}\iota(c^{o(c)/2})=
\iota(c^{o(c)/2})^{i_0}\iota(c^{o(c)/2})=1,$$
 the first equality since $o(c^i)$ is odd, the second since $\tau(c)$
 satisfies $(\flat)$, and the third since $i_0$ is odd.
\end{proof}
We now prove Proposition \ref{lifting2} type  by type. Note that for types $B$, 
$C$, and $D$ we use the labeling of \cite{chevie} which is reversed from the Bourbaki labeling.
\newcommand\nnode[1]{{\kern -0.6pt\mathop\bigcirc\limits_{#1}\kern -1pt}}
\newcommand\edge{{\vrule width10pt height3pt depth-2pt}}
\newcommand\vertbar[2]{\rlap{\kern4pt\vrule width1pt height17.3pt depth-7.3pt}
 \rlap{\raise19.4pt\hbox{$\kern -0.4pt\bigcirc\scriptstyle#2$}}
                 \nnode{#1}}
\subsubsection{Type $A_n$}\label{type A}$\nnode{1}\edge\nnode{2}\cdots\nnode{n}$

We may assume that  $\bG=\Sl_{n+1}$ and  let $S=\{s_1,\ldots,s_n\}$.  Then
$c=s_1\cdots s_n=w_0 w_I$ for $I=\{2,\ldots,n\}$ (to simplify the notation, we
will represent subsets of $S$ using only their indices, so $\{2,\ldots,n\}$
represents $\{s_2,\ldots,s_n\}$) is a generator of $\CA$.
If  $\bc$ is the lift to $\bW$ of  $c$ (resp.\ $\bw_0$ the lift of $w_0$) we
have    $\bc^{n+1}=\bw_0^2$.   Thus    applying   $\ts$    we   get
$\sigma(c)^{n+1}=\sigma(w_0)^2=\rho^\vee$, the last equality
by Corollary \ref{sigma(involution)}(ii).  
According  to  \cite[planche
I]{bbk}  we have $2\rho^\vee=\sum_{i=1}^n i(n-i+1)\alpha_i^\vee$ from which
it  follows that if $n$ is  even we have $\rho^\vee\in \BZ\Phi^\vee=Y(\bT)$
and  if $n$ is odd we have $\rho^\vee\equiv (\sum_{i=1}^n i\alpha_i^\vee)/2
\pmod{Y(\bT)}$.

Let us denote by $n_p$ and $n_{p'}$ respectively the $p$-part and the
$p'$-part of an integer $n$.
The group $Z\bG$ is generated by $\iota(c^{(n+1)_p})$, the image of
the  coweight of index $(n+1)_p$  in $Y(\bT)\otimes (\BQ/\BZ)_{p'}$, which,
still  according  to  \cite[planche  I]{bbk}  is  equal  to  $(\sum_{i=1}^n
-i\alpha_i^\vee)/(n+1)_{p'}\pmod{Y(\bT)}$. It follows that
$$\sigma(c)^{n+1}=\sigma(c)^{o(c)}=\rho^\vee=
\begin{cases}\iota(c^{o(c)_p})^{o(c)_{p'}/2}=\iota(c^{o(c)/2})&\text{when
$o(c)$    is   even,}\\    1&\text{otherwise,}\end{cases}$$ 
thus $\tau(c)=\sigma(c)$ has property $(\flat)$.
For the other elements $c^i$ of $\CA$ we take $\tau(c^i)$ according to Lemma
\ref{c^i}. Note that these elements commute since $\rho^\vee$ is either
trivial (when $n$ is even) or the element $-1\in Z\bG$ (when $n$ is odd).

\subsubsection{Type $B_n$}\label{type B}
$\nnode{1}\rlap{\kern 2pt$<$}{\rlap{\vrule width10pt height2pt depth-1pt}
\vrule width10pt height4pt depth-3pt}
\nnode{2}\edge\nnode{3}\cdots\nnode{n}$

We may assume that $\bG=\Spin_{2n+1}$ and  let $S=\{s_1,\ldots,s_n\}$. The group $\CA$ is
generated by $c=w_0w_I$ where $I=\{1,\ldots,n-1\}$. Since $c$ is
an involution we get $\sigma(c)^2=\rho^\vee-\rho^\vee_I$ by Corollary
\ref{sigma(involution)}(i). We use the notation and the 
formulae of \cite[Planche III]{bbk} except that the indexing of the simple roots is
reversed. We get 
$\rho^\vee=\sum_{i=1}^{i=n} i\varepsilon_i$ and 
$\rho_I^\vee=\sum_{i=2}^{i=n} (i-1)\varepsilon_i$, whence
$\sigma(c)^2=\sum_{i=1}^{i=n}\varepsilon_i\equiv
\alpha_1^\vee/2\equiv\varpi_n^\vee\equiv\iota(c)\pmod{Y(T)}$, whence Lemma
\ref{lifting2} in
this case with $\tau=\sigma$.
\subsubsection{Type $C_n$}\label{C}
$\nnode{1}\rlap{\kern 2pt$>$}{\rlap{\vrule width10pt height2pt depth-1pt}
\vrule width10pt height4pt depth-3pt} 
\nnode{2}\edge\nnode{3}\cdots\nnode{n}$

We may assume that $\bG=\Sp_{2n}$ and  let $S=\{s_1,\ldots,s_n\}$. On $V=k^{2n}$ with basis
$e_n,\ldots,e_1,e'_1,\ldots,e'_n$ we can realise $\bG$ as the group of isometries of the
symplectic form
$\langle e_i,e_j\rangle=\langle e'_i,e'_j\rangle=0$ and
$\langle e_i,e'_j\rangle=-\langle e'_i,e_j\rangle=\delta_{i,j}$. 
Taking for $\bT$ the diagonal torus,
the set $S$ consists of the permutation matrices $s_1=(e_1,e'_1)$ and
$s_i=(e_i,e_{i+1})(e'_i,e'_{i+1})$. The group $\CA$ is generated by 
the involution $c=w_0w_I$ where $I=\{2,\ldots n\}$. We choose the matrix $\tau(c)=\begin{pmatrix}
0&-\Id_n\\ \Id_n&0\end{pmatrix}$ as a representative of $c$.
We have $\tau(c)^2=-\Id_{2n}$ which is
Lemma \ref{lifting2} since $-\Id_{2n}=\iota(c)$ is the generator of the
centre of $\bG$ (the centre has order 2). Note that
$\tau(c)\ne\sigma(c)$ in general since a computation shows that $\sigma(c)^2=
(-1)^n\Id_{2n}$.

\subsubsection{Type  $D_{2r}$} \label{D even}
$\nnode{1}\edge\vertbar{3}{2}\edge\nnode{4}\cdots\nnode{2r}$

\newcommand\xa{A}
\newcommand\xb{B}
\newcommand\bxa{{\bf A}}
\newcommand\bxb{{\bf B}}
We may assume that $\bG=\Spin_{4r}$ and  let $S=\{s_1,\ldots,s_{2r}\}$.
The elements of $Z\bG$, viewed as a subgroup of $Y(\bT)\otimes(\BQ/\BZ)_{p'}$,
are the fundamental weights $1,\varpi^\vee_1,\varpi^\vee_2$
and $\varpi^\vee_n$ $\pmod{Y(\bT)}$, where $n=2r$. The non-trivial elements of
$\CA$ are the involutions
$a=w_0w_{\{2,\ldots, n\}}, b=w_0w_{\{1,3,\ldots, n\}}$ (they correspond to $c$ and
$c'$ below Lemma \ref{lifting2}) and $c=w_0w_{\{1\ldots n-1\}}$.
Using Corollary \ref{sigma(involution)}(i) we find that
$\sigma(a)^2=\varpi^\vee_1$, $\sigma(b)^2=\varpi_2^\vee$ and $\sigma(c)^2=
\varpi_n^\vee$.
We denote by $\ba,\bb, \bc$ the lifts to $\bW$ of $a,b,c$. Note that
$ab=ba=c$ but that these relations do not hold for their lifts to $\bW$, 
in particular $\ba$ and $\bb$ do not commute. However, we will show
that $\bc\bb$ (which lifts $a$) and $\bc\ba$ (which lifts $b$) commute.
We will take $\tau(a)=\ts(\bc\bb),
\tau(b)=\ts(\bc\ba)$ and $\tau(c)=\ts(\bc\bb\bc\ba)$. Let $\xa=s_2s_3\cdots
s_n$ and $\xb=s_1 s_3\cdots s_n$. We first observe that
$c=\xb\inv\xa=\xa\inv\xb$; to check the first equality we can
observe that $\alpha_1,\ldots,\alpha_{n-1}$ are all sent to positive roots
by $\xb\inv\xa$, so that $\xb\inv\xa$ is $W_{\{1,\ldots,n-1\}}$-reduced,
and it has length $2n-2=l(w_0)-l(w_{\{1\ldots n-1\}})$ (one has
$l(\xb\inv\xa)=l(\xb\inv)+l(\xa)$ since each one of $\xb$ and $\xa$ has
exactly one reduced expression and their first terms differ); thus
$\xb\inv\xa=c$,
the only $W_{\{1,\ldots,n-1\}}$-reduced element of its length. The additivity of the lengths
shows also that the element $\xb\inv\xa$ lifts to the braid group as $\tilde\bxb\bxa$.
Applying to $\xb\inv\xa=c$ the  symmetry which exchanges
$s_1$ and $s_2$ and leaves $c$ invariant, we get $\xb\inv\xa=\xa\inv\xb$;
which lifts to the braid group as
$\tilde\bxb\bxa=\tilde\bxa\bxb$. Let $I=\{3,\ldots,n\}$. In the
subgroup of type $A$ given by $W_{\{2\ldots n\}}$ we have the equalities
$w_{\{2\ldots n\}}=w_I\xa=\xa\inv w_I$ which lift to the braid group as
$\bw_{\{2\ldots n\}}=\bw_I\bxa=\tilde\bxa\bw_I$ and similarly in the group of type $A$
given by $W_{\{1,3\ldots n\}}$ we have $w_{\{1,3\ldots n\}}=w_I\xb=\xb\inv w_I$ which lifts
to the braid group as $\bw_{\{1,3\ldots n\}}=\bw_I\bxb=\tilde\bxb\bw_I$
whence,  if $\bw_I$ and  $\bw_0$ are the  lifts to $\bW$  of $w_I$ and
$w_0$ (note that $\bw_0$ is central), and $\bxa$ and $\bxb$ the lifts of
$\xa$ and $\xb$,
the equalities $\ba=\bw_0\bxa\inv\bw_I\inv$ and
$\bb=\bw_0\bxb\inv\bw_I\inv$. Since we have
$\bc=\tilde\bxb\bxa=\tilde\bxa\bxb$, it follows that $\bc\ba=
\tilde\bxb\bw_0\bw_I\inv$ and $\bc\bb=\tilde\bxa\bw_0\bw_I\inv$,
thus
$\bc\ba\bc\bb=\tilde\bxb\bw_0\bw_I\inv\tilde\bxa\bw_0\bw_I\inv=
\tilde\bxb\bxa\bw_0\bw_I\inv\bw_0\bw_I\inv=
\tilde\bxa\bxb\bw_0\bw_I\inv\bw_0\bw_I\inv=
\tilde\bxa\bw_0\bw_I\inv\tilde\bxb\bw_0\bw_I\inv=
\bc\bb\bc\ba$. 

We now show that $\tau(a)$ and $\tau(b)$  satisfy condition $(\flat)$ of
Lemma \ref{lifting2} (where we have called $a$ and $b$ the generators of
$\CA$  instead of $c$ and $c'$). 
Using the equalities
$\bw_0\bw_I\inv=\bxb\bb=\bb\tilde\bxb$ we get
$(\bc\ba)^2=\tilde\bxb\bw_0\bw_I\inv\bw_0\bw_I\inv\bxb=
\tilde\bxb\bxb\bb\bb\tilde\bxb\bxb$. If we apply $\ts$ and use that
$\ts(\bb^2)=\sigma(b)^2$ is central, we get
$\ts((\bc\ba)^2)=\ts(\tilde\bxb\bxb)\ts(\bxb\tilde\bxb)\sigma(b)^2$. Using
twice Proposition \ref{adams-vogan} we get
$\ts(\tilde\bxb\bxb)\ts(\bxb\tilde\bxb)=
\rho^\vee-(\xb(\rho^\vee)+\xb\inv(\rho^\vee))/2$.

\begin{lemma}
We have $\xb(\rho^\vee)=\rho^\vee-(n-1)\alpha_1-\sum_{i=3}^{i=n}(n-i+1)\alpha_i$ and
$\xb\inv(\rho^\vee)=\rho^\vee-\alpha_1-\sum_{i=3}^{i=n}(i-1)\alpha_i$.
\end{lemma}
\begin{proof}
We use repeatedly the formulae $s_i\rho^\vee=\rho^\vee-\alpha_i^\vee$ and
$s_i(\alpha_{i+1})=s_{i+1}(\alpha_i)=\alpha_i+\alpha_{i+1}$ for $i>2$, as well as
$s_1(\alpha_3)=s_3(\alpha_1)=\alpha_1+\alpha_3$.
\end{proof}
Using this lemma we get 
$\ts(\tilde\bxb\bxb)\ts(\bxb\tilde\bxb)=(n/2)\sum_{i\neq
2}\alpha_i\equiv0\pmod{Y(\bT)}$,
the last equality since $n/2=r\in\BZ$, 
so finally $\tau(b)^2=\ts((\bc\ba)^2)=\sigma(b)^2=\varpi_2^\vee=\iota(b)$.
By symmetry we get the similar result for $a$, whence Lemma
\ref{lifting2} in this case.
\subsubsection{Type $D_{2r+1}$}\label{D odd}
$\nnode{1}\edge\vertbar{3}{2}\edge\nnode{4}\cdots\nnode{2r+1}$

We may assume that $\bG=\Spin_{4r+2}$ and  let $S=\{s_1,\ldots,s_{2r+1}\}$.
The elements of $Z\bG$ are still $1,\varpi^\vee_1,\varpi^\vee_2$
and $\varpi^\vee_n$ $\pmod{Y(\bT)}$, where $n=2r+1$ but this time they
form a cyclic group of order 4 generated by $\varpi^\vee_1$ or
$\varpi^\vee_2$. The conjugation by $w_0$ induces the automorphism which
exchanges $s_1$ and $s_2$ and leaves the other $s_i$ invariant.
We keep from the case $D_{2r}$ the notation $I$, $\bxa$, $\bxb$, 
$\bw_0$, $\bw_I$ and $\ba$. 

\begin{lemma}\label{tilde-invariance} The braid
$(\tilde\bxb\ba\tilde\bxb\inv)^2$ is invariant by $\sim$.
\end{lemma}
\begin{proof}
We will use the facts that the conjugation by $\bw_0$ fixes $\bw_I$ and
exchanges $\bxa$ and $\bxb$ and also that
$\bw_I\bxa=\tilde\bxa\bw_I$ and
$\bw_I\bxb=\tilde\bxb\bw_I$. We have
$\ba=\bw_0\bxa\inv\bw_I\inv$, hence $\ba^2=\bw_0^2\bxb\inv\bw_I\inv\bxa\inv\bw_I\inv=
\bw_0^2\bxb\inv\bw_I\inv\bxa\inv\bw_I\inv$.
Using that $\bw_0^2$ is central we deduce
$\tilde\bxb\ba^2\tilde\bxb\inv=\bw_0^2\tilde\bxb\bxb\inv\bw_I\inv\bxa\inv\bw_I\inv\tilde\bxb\inv$
We can forget $\bw_0^2$ which is central and $\sim$-invariant. 
Hence we have to prove that
$\tilde\bxb\bxb\inv\bw_I\inv\bxa\inv\bxb\inv\bw_I\inv=
\tilde\bxb\bxb\inv\bw_I^{-2}\tilde\bxa\inv\tilde\bxb\inv$
is $\sim$-invariant. For simplicity we show the property on the inverse
$\tilde\bxb\tilde\bxa\bw_I^2\bxb\tilde\bxb\inv$.
We have $\bw_I\bxb=\bw_{\{1,3\ldots n\}}$, the conjugation by which flips the Dynkin diagram
supported $1,3\ldots n$. Thus $\bw_I\bxb\tilde\bxb\inv=\bxb\inv\bw_I\bxb$,
whence
$\tilde\bxb\tilde\bxa\bw_I^2\bxb\tilde\bxb\inv=
\tilde\bxb\tilde\bxa\bw_I\bxb\inv\bw_I\bxb=
\tilde\bxb\bw_I\bxa\bxb\inv\bw_I\bxb$. We now use the fact that 
$\tilde\bxa\bxb=\tilde\bxb\bxa$ which shows that $\bxa\bxb\inv$
is $\sim$-invariant. Since $\bw_I$ is $\sim$-invariant, we get that 
$\tilde\bxb\bw_I\bxa\bxb\inv\bw_I\bxb$ is also
$\sim$-invariant, whence the lemma. 
\end{proof}
Lemma \ref{tilde-invariance} allows us to compute 
$\ts(\tilde\bxb\ba\tilde\bxb\inv)^4$ using Proposition \ref{adams-vogan}.
\begin{lemma}\label{computation Dimpair}
We have 
$\ts(\tilde\bxb\ba\tilde\bxb\inv)^4=(\rho^\vee-(\xb\inv
a\xb)^2(\rho^\vee))/2\equiv(\alpha_1^\vee+\alpha_2^\vee)/2\pmod{Y(\bT)}$.
\end{lemma}
\begin{proof}We put $n=2r+1$ and
use the basis $(\varepsilon_i)_{i=1,\ldots,n}$ of $Y(\bT)$ (see
\cite[Planche IV]{bbk} but with the numbering as in the diagram above).
In this basis $\xb$ maps $\varepsilon_i$ to $\varepsilon_{i+1}$ for
$i=2,\ldots,n-1$, and maps $\varepsilon_1$ to $-\varepsilon_2$ and
$\varepsilon_n$ to $-\varepsilon_1$. Using
 $a^2=w_0^2w_{\{1,3,\ldots,n\}}w_{\{2,3,\ldots,n\}}=w_{\{1,3,\ldots,n\}}w_{\{2,3,\ldots,n\}}$
we see that $a^2$ changes the signs of $\varepsilon_1$ and $\varepsilon_n$ and
fixes the other $\varepsilon_i$. Thus $B\inv a^2 B$ changes $\varepsilon_1$,
$\varepsilon_{n-1}$ and $\varepsilon_n$ into their opposites and fixes the other
$\varepsilon_i$. We have
$\rho^\vee=\sum_{i=2}^{i=n}(i-1)\varepsilon_i$.
Thus we get $(\rho^\vee-\xb\inv
a^2\xb(\rho^\vee))/2=(n-1)\varepsilon_n+(n-2)\varepsilon_{n-1}=
(n-1)\alpha_n^\vee+(2n-3)(\alpha_{n-1}^\vee+\cdots+\alpha_3^\vee)+
((2n-3)/2)(\alpha_1^\vee+\alpha_2^\vee)\equiv(\alpha_1^\vee+\alpha_2^\vee)/2\pmod{Y(\bT)}$.
\end{proof}
Lemma \ref{computation Dimpair} gives Lemma \ref{lifting2} for $a$ with
$\tau(a)=\sigma(a)$: since $\ts(\tilde\bxb \ba\tilde\bxb\inv)^4$ is central, we have
$\ts(\ba)^4=\ts(\tilde\bxb \ba\tilde\bxb\inv)^4$, hence $\tau(a)^4=
(\alpha_1^\vee+\alpha_2^\vee)/2=\varpi_n^\vee=2\varpi_1^\vee=\iota(a^2)$.

Using Lemma \ref{c^i},
for $a^2$ we can take $\tau(a^2)=\ts(\ba)^2=\ts(\ba^2)$ and for
$a^3$ we can take $\tau(a^3)=\ts(\ba)^3=\ts(\ba^3)$.

\subsubsection{Type $E_6$} 
$\nnode1\edge\nnode3\edge\vertbar42\edge\nnode5\edge\nnode6$

Here $\CA$ is generated by $c=w_0w_I$  where $I=\{1,\ldots,5\}$. Lifting to
$B(W)$ we check using \cite{chevie} that $\bc^3=\bw_0^2\bw_J^{-2}$ where 
$J=\{2,\ldots,5\}$ is of type $D_4$: here is the code with the \Chevie\ package 
of Julia:
\begin{verbatim}
julia> W=coxgroup(:E,6);B=BraidMonoid(W);

julia> c=B(longest(W))/B(longest(W,1:5))
1342543165423456

julia> c^3==B(longest(W))^2/B(longest(W,2:5))^2
true
\end{verbatim}
Applying $\ts$ it follows that
$\sigma(c^3)=\rho^\vee-\rho^\vee_J$ by Corollary \ref{sigma(involution)}(i).
Looking  at  \cite[Planche IV and
Planche  V]{bbk} we see that 
both $\rho^\vee$  and $\rho^\vee_J$ are  in $Y(\bT)$. Thus
$\sigma(c^3)=1$  which agrees with $\iota(c^3)=1$, whence Lemma \ref{lifting2} for $c$ with
$\tau(c)=\sigma(c)$. For $c^2$ we can take $\tau(c^2)=\sigma(c)^2$ according to
Lemma \ref{c^i}.

\subsubsection{Type $E_7$}\label{E7}
$\nnode1\edge\nnode3\edge\vertbar42\edge\nnode5\edge\nnode6\edge\nnode7$

Here, $\CA$ is generated by $c=w_0w_I$  where $I=\{1,\ldots,6\}$. Since $c$
is an involution, Corollary \ref{sigma(involution)}(i) gives
$\sigma(c)^2=\rho^\vee-\rho_I^\vee$.
We have seen (in case $E_6$) that $\rho_I^\vee\in Y(\bT)$. From
\cite[planche VI]{bbk} we find that
$\rho^\vee\equiv(\alpha^\vee_2+\alpha^\vee_5+\alpha^\vee_7)/2 \pmod{Y(\bT)}$.
We also find that $\varpi^\vee_7=\iota(c)$
has the same value, whence Lemma \ref{lifting2} in this case with $\tau=\sigma$.

\subsubsection*{Types $E_8$, $F_4$ and $G_2$} There is nothing to do since both
the groups $\CA$ and $Z\bG$ are trivial.
\section{Finite reductive groups}
We assume now that $\bG$ is a connected reductive group defined over an algebraic
closure $k$ of a finite field $\Fq$ and that $F$ is a Frobenius root, that is 
an endomorphism such that some power is a Frobenius endomorphism attached to an $\Fq$-structure on $\bG$. 

We assume further that $s$ is a semisimple element such that $C_\bG(s)$ is $F$-stable (which is
the case in particular if $s\in\bG^F$; however our more general assumption
turns out to be more suited to a reduction argument); the sequence
$$1\to C_\bG(s)^{0F}\to C_\bG(s)^F\to A_\bG(s)^F\to 1\eqno(\sharp)$$ 
is exact, since in the Galois cohomology exact sequence the next term
$H^1(F,C_\bG(s)^0)$ is trivial by Lang's theorem (see \cite[4.2.9]{DM}).

We recall that a quasi-split maximal torus of $\bG$ is an $F$-stable maximal
torus which is in
an $F$-stable Borel subgroup. The quasi-split maximal tori of $\bG$ are conjugate
under $\bG^F$.
\begin{lemma}\label{sharpsharp}
Let $\bT$ be a quasi-split maximal torus of $C_\bG(s)^0$; then 
the following sequence is exact
$$1 \to N_{C_\bG(s)^0}(\bT)^F \to  N_{C_\bG(s)}(\bT)^F\to A_\bG(s)^F\to1.\eqno{(\sharp\sharp)}$$
\end{lemma}
\begin{proof} The sequence $(\sharp)$ restricted to $N_{C_\bG(s)}(\bT)$ gives the sequence
$(\sharp\sharp)$ if we can prove the surjectivity of the
third morphism.  Let us show this. An element $a\in
A_\bG(s)^F$ is the image of some $x\in C_\bG(s)^F$. Since $\bT$ is quasi-split and $x$
is $F$-stable, the torus $\lexp x\bT$ is also quasi-split, so that $\lexp x\bT=\lexp y\bT$ with
$y\in C_\bG(s)^{0F}$. Then $y\inv x$ is in $N_{C_\bG(s)}(\bT)^F$ and has image $a$ in
$A_\bG(s)^F$.
\end{proof}
\begin{lemma}\label{AW(s)F}
Let $\bT$ be a quasi-split maximal torus of $C_{\bG}(s)$, let 
$W(s)=N_{C_\bG(s)}(\bT)/\bT$ and 
$W^0(s)=N_{C_\bG(s)^0}(\bT)/\bT$. Let $\Phi^+(s)$ be a system of positive
roots of $C_\bG(s)^0$ corresponding to an $F$-stable Borel subgroup
containing $\bT$. We have a split exact sequence
$$1\to W^0(s)^F\to W(s)^F\to A_\bG(s)^F\to 1$$
 where $A_\bG(s)^F$ lifts to $A_W(s)^F=\{w\in W(s)^F\mid \lexp
 w\Phi^+(s)=\Phi^+(s)\}$.
\end{lemma}
\begin{proof}
Since $\bT$ is connected the quotient by $\bT^F$ of the exact sequence $(\sharp\sharp)$
gives the exact sequence of the lemma. 
By Proposition \ref{AW(s)}
an element of $A_\bG(s)^F$ lifts to a unique element of $A_W(s)$ which is thus $F$-fixed,
whence the splitting of the exact sequence.
\end{proof}
We will show that the
sequence $(\sharp\sharp)$ of Lemma \ref{sharpsharp} splits. This implies that
the sequence $(\sharp)$ splits.

\subsection{Reduction to semisimple groups}
\begin{proposition}\label{derivedF}
Let $s\in\bG$ be a semisimple element such
that $C_\bG(s)$ is $F$-stable and let $\bT$ be a quasi-split maximal torus of 
$C_\bG(s)^0$. Let
$\bG'$ be the derived group of $\bG$ and let $\bT'=\bT\cap\bG'$. If for any
$s'\in\bT'$  such that $C_{\bG'}(s')$ is  $F$-stable and $\bT'$ is quasi-split in
$C_{\bG'}(s')^0$, we have a semidirect product
decomposition $N_{C_{\bG'}(s')}(\bT')^F=N_{C_\bG'(s')^{0}}(\bT')^F\rtimes A'_0$,
then we have a semidirect decomposition
$N_{C_\bG(s)}(\bT)^F=N_{C_\bG(s)^0}(\bT)^F\rtimes A_0$.
\end{proposition}
\begin{proof}
As  in  the  proof  of  Proposition  \ref{derived},  we  write $s=s'z$ with
$s'\in\bG'$  and  $z\in(Z\bG)^0$. 
Since  $C_{\bG'}(s')=C_\bG(s)\cap\bG'$, this  group is  $F$-stable. Thus, by
assumption, we have a decomposition
$N_{C_{\bG'}(s')}(\bT')^F=N_{C_{\bG'}(s')^0}(\bT')^F\rtimes  A_0$.
By  the proof  of Proposition  \ref{derived}, the  group $A_0$ does not meet
$N_{C_\bG(s)^0}(\bT)^F$ thus $N_{C_\bG(s)}(\bT)^F\supset N_{C_\bG(s)^0}(\bT)^F\rtimes A_0$.
Since $W(s)=W(s')$ and $W^0(s)=W^0(s')$, Lemma \ref{AW(s)F} shows that
$|A_\bG(s)^F|=|A_{\bG'}(s')^F|=|A_0|$. This shows that
$|N_{C_\bG(s)}(\bT)^F|/|N_{C_\bG(s)^0}(\bT)^F|=|A_0|$, whence the result.
\end{proof}

\begin{remark}  
Note that the above proof of Proposition \ref{derivedF} would not have worked with the
assumption   $s\in\bG^F$  since   there  exist   examples  where  for  some
$s\in\bG^F$  there is no $z\in Z\bG$  such that $sz\in\bG^{\prime F}$. For example, take
$\bG=\GL_2$  with  $F$  the  natural $\BF_q$-structure. Take $s=\diag(a,b)$
with  $a,b\in\BF_q$  such  that  $ab$  is  not  a  square  in  $\BF_q$. Let
$z=\diag(c,c)$.  The  condition  that  $z=\lexp  Fz$ is $c\in\BF_q$ and the
condition  $sz\in\bG'$ is $\det(sz)=1$ or  $c^2ab=1$ which contradicts that
$ab$ is not a square.
\end{remark}

\subsection{Tits lifting}
\begin{proposition} \label{Tits F-fixed} 
Let $\bT$ be an $F$-stable maximal torus of $\bG$ and assume that there exists
a Borel subgroup $\bB$ containing $\bT$ such that $F$ acts by a diagram
automorphism $\varphi$  on the corresponding Coxeter system $(W,S)$.
Then there exists a $(\varphi,F)$-equivariant Tits lifting $W\xrightarrow\sigma N_\bG(\bT)$.
\end{proposition}
\begin{proof}
Let $\Pi$ be the basis of roots corresponding to the choice of $\bB$.
By \cite[proof of 3.3]{Tits}, if for each $\alpha\in\Pi$ we choose a
representative $\dot s_\alpha$ of $s_\alpha$
in the subgroup $\bG_\alpha$ of $\bG$ generated by $\bU_\alpha$ and
$\bU_{-\alpha}$ (the one-dimensional unipotent subgroups corresponding to
the roots $\alpha$ and $-\alpha$ normalised by $\bT$), the elements $\{\dot s_\alpha\}_{\alpha\in\Pi}$
satisfy the braid relations, hence define a Tits lifting by lifting 
reduced expressions.  Let $\{\bG_{\alpha_1},\ldots,\bG_{\alpha_d}\}$ be
the orbit of $\bG_{\alpha_1}$ under $F$ for $\alpha_1\in\Pi$.
Since $s_{\alpha_1}$ is $\varphi^d$-fixed,
we can choose an $F^d$-fixed $\dot s_{\alpha_1}$ by Lang's theorem used in
$\bG_{\alpha_1}\cap\bT$, a maximal torus of $\bG_{\alpha_1}$. Then
$\dot s_{\alpha_1},F(\dot s_{\alpha_1}),\ldots,F^{d-1}(\dot s_{\alpha_1})$ 
defines a $(\varphi,F)$-equivariant lifting for the orbit of $\bG_{\alpha_1}$ and
we make a similar choice for the other orbits.
\end{proof}
Note that the above $\sigma$ defines a $(\varphi,F)$ equivariant
$\ts: B(W)\to N_\bG(\bT)$.

\subsection{Case of semisimple groups}\label{case semisimple}
Let us first describe a general construction in a reductive group
$\bG$ with a Frobenius root $F$.

Let $\bT$ be an $F$-stable maximal torus of $\bG$ such that there exists
a Borel subgroup containing $\bT$ such that $F$ acts by a diagram automorphism
$\varphi$ on the corresponding Coxeter system $(W,S)$.
Note that this does {\em not} mean that there is an $F$-stable Borel subgroup
containing $\bT$. For example, in an quasi-simple group of type
$\lexp 2A_n,\lexp 2E_6,\lexp 2D_{2r+1}$ we may take a torus of type $w_0$
for which $F$ acts trivially on $W$.

For a semisimple element $s_1\in\bG$
such that $C_\bG(s_1)$ is $F$-stable, we want to lift $A_\bG(s_1)^F$.
Let $\bT_1$ be an $F$-stable maximal torus of $C_\bG(s_1)$; it contains $s_1$ and
there exists $g\in\bG$ which conjugates
$\bT_1$ to $\bT$. Then $g$ conjugates the quadruple $(\bG,\bT_1,F,s_1)$
to a quadruple $(\bG,\bT,\dot v_0 F,s_0)$ where $\dot v_0=g\cdot\lexp Fg\inv\in
N_\bG(\bT)$ and $s_0\in\bT$. 

The element $g$ is determined up to left multiplication by $N_\bG(\bT)$. Changing
$g$ to $\dot v g$ for $\dot v\in N_\bG(\bT)$ changes
the quadruple $(\bG,\bT,\dot v_0 F,s_0)$ into $(\bG,\bT,\dot v\dot v_0 F\dot v\inv,
\lexp vs_0)$, where $v$ is the image of $\dot v$ in $W$. 

As in the first section we represent $s_0$ by $\lambda_0\in Y(\bT)\otimes\BQ$.
Replacing $s_0$ by a $W$-conjugate $\lexp vs_0$ and $\lexp v\lambda_0$ by 
a $Y(\bT)$-translate we can get a $\lambda\in\CC$ representing $\lexp vs_0$.
We set $s=\lexp v s_0$.

Conjugating $\bT_1$ by $h\in C_\bG(s_1)^0$ to another $F$-stable maximal torus containing $s_1$, 
amounts to changing $\dot v \dot v_0 \cdot\lexp F{\dot v\inv}$ into $\dot u \dot v
\dot v_0 \cdot\lexp F{\dot v\inv}$, where $\dot
u=\lexp{\dot vg}(h\inv\cdot\lexp Fh) \in N_{C_\bG(s)^0}(\bT)$; indeed
$\dot vgh\inv\cdot\lexp Fh\lexp F (\dot vg)\inv=
\lexp {\dot vg}(h\inv\cdot\lexp Fh) \dot v g\lexp F (\dot vg)\inv=\dot u \dot v
\dot v_0 \cdot\lexp F{\dot v\inv}$.
 By Lang's theorem in $C_\bG(s_1)^0$, the element $\dot u$ can be any element of
$N_{C_\bG(s)^0}(\bT)$ when $h$ varies.

Since $C_\bG(s_1)$ is $F$-stable, $C_\bG(s)$ is $\dot v \dot v_0 F{\dot v\inv}$-stable.

We choose $\Phi^+(s)$ as determined by the same basis
$(\Pi\cap \Phi(s))\cup\{\tilde\alpha_i\mid
\langle\lambda,\tilde\alpha_i\rangle=-1\}$ of $\Phi(s)$ as in the first
section.

If $F$ is a Frobenius endomorphism, $\varphi$ induces an automorphism of $\Phi(s)$; for a
general Frobenius root it may happen that $\varphi$ sends a root to a positive rational
multiple of a root.

For an appropriate choice of $\dot u$, hence of $h$, the image $w\in W$ of 
$\dot w:=\dot u \dot v \dot v_0 \cdot\lexp F{\dot v\inv}\in N_\bG(\bT)$ 
is such that $w\varphi$ stabilises the positive cone $\BQ^+\Phi^+(s)$.

We  have thus managed to get a quadruple $(\bG,\bT,\dot wF,s)$ conjugate to
$(\bG,\lexp  h\bT_1,F,s_1)$ such  that $C_\bG(s)$  is $\dot  w F$-stable and
the  positive cone  $\BQ^+\Phi^+(s)$ is $w\varphi$-stable
and $s$ is represented by
$\lambda\in  \CC$. 
Finally note that by multiplying the conjugating element
$g$  on  the  left  by  $t\in\bT$  we  multiply  $\dot  w$  on  the left by
$t\cdot\lexp{wF}t\inv$,  that is  by any  element of  $\bT$. So  we can take any
representative $\dot w$ of $w$.

Note that $w\varphi$ stabilising $\BQ^+\Phi^+(s)$ implies that
$F$ stabilises a Borel subgroup of $C_\bG(s_1)$ containing $\lexp h\bT_1$, 
thus $\lexp h\bT_1$ is quasi-split in $C_\bG(s_1)$. 
Thus if we lift $A_\bG(s)^{\dot wF}$, which we will do in Proposition \ref{group lifting'},
by geometric conjugacy by the element $g$
this lifts $A_\bG(s_1)^F$ and shows that the sequence $(\sharp\sharp)$ splits.

The sequence $1\to Z\to\Gsc\xrightarrow\pi\bG\to 1$ below \ref{decomposition} can be
chosen $F$-equivariant (see \cite[9.16]{steinberg}). Since $F$ acts by a diagram automorphism on $W$, 
the decomposition $\Gsc=\bG_1\times\cdots\times\bG_n$ in the proof of
Proposition \ref{lifting} is $F$-stable, and
the corresponding decomposition $\CA=\CA_1\times\cdots\times\CA_n$ is $\varphi$-stable.
We decompose $\Gsc$ into a product $\bH_1\times\cdots\times\bH_m$ where each
$\bH_i$ is the product of one orbit of $\bG_i$ under $F$ and we decompose $\CA$
accordingly into $\CA=\CB_1\times\cdots\times\CB_m$.

\newcommand\tA{{\tilde{\CA}}}
The group $A_W(s)^{w\varphi}$ is a subgroup of $\CA$. 
Let $\pi_i$ be the projection $\CA\to\CB_i$. We will lift to 
$N_\Gsc(\Tsc)^{\dot wF}$ a larger group, 
the group $\tA=\prod_i\pi_i(A_W(s)^{w\varphi})$, so that the elements of
the lift satisfy $(\flat)$ which will imply
that the lift composed with $\pi$ is a group morphism.
In particular
\begin{proposition}\label{group lifting'} 
There exists a representative $\dot w\in N_\bG(\bT)$ of $w$ such that
we can lift $A_W(s)^{w\varphi} $ by a group morphism to $N_\bG(\bT)^{\dot wF}$.
\end{proposition}
\begin{proof}
If  $\ch k=2$,  we may,  as in  the first  section lift  $W$ by a Tits
lifting $\sigma$ which is a group morphism in that case; and this lifting may
be chosen
$(\varphi,F)$-equivariant by Proposition \ref{Tits F-fixed}. The centraliser of $w\varphi$ lifts
to the centraliser of $\dot wF$ if we take $\dot w=\sigma(w)$. This lifts
$A_W(s)^{w\varphi}$.

We assume from now on that $\ch k$ is odd.

A part of the proof, that of Lemma \ref{lifting'} in the case of simply 
connected quasisimple groups, will be done in the next subsection.
\begin{lemma} \label{lifting'} 
For any $w\in W$ and any $w\varphi$-fixed subgroup $\tA$ of $\CA$ such that
$\tA=\Pi_i\pi_i(\tA)$,
there exists a representative $\dot w\in N_\Gsc(\Tsc)$ of $w$ and a lift
$\tau$ of $\tA$ to a set of commuting elements  of 
$N_\Gsc(\Tsc)^{\dot w F}$ such  that for each  $a\in\tA$ the element
$\tau(a)$ has property $(\flat)$.
\end{lemma}
\begin{proof}
We show that, assuming Lemma \ref{lifting'} holds for the product $\bH_i$ of an orbit under $F$ of
quasi-simple groups (which we will show in the next subsection, see Lemma \ref{orbit} and
Proposition \ref{lifting''}), it holds for $\Gsc$.
If there exists $\tau$ as in the statement for each component
$\pi_i(\tA)$, then, by Lemma \ref{product},
defining $\tau$ component by component yields a $\tau$ satisfying $(\flat)$
since the components $\pi_i(\tA)$ commute with one another. 
\end{proof}
We  now deduce  Proposition \ref{group  lifting'} from Lemma \ref{lifting'}
following  the same steps  as for deducing  Proposition \ref{group lifting}
from Proposition \ref{lifting}: we consider first $\pi(\tau(\tA))$ which,
by condition $(\flat)$, is a set of commuting elements in
$N_\bG(\bT)^{\pi(\dot  w)F}$ such that for any $a\in \tA$ the order
$o(a)$   is  equal  to  the   order  of  $\pi(\tau(a))$,  then  decomposing
$\tA$   into  a  direct  product  of  cyclic  groups  as  in  Lemma
\ref{tau2},  we define a lift which  is a group morphism from $\tA$
to $N_\bG(\bT)^{\pi(\dot w)F}$.
\end{proof}
Note that, even if $F$ is a Frobenius endomorphism, $\iota:\CA_{p'}\to Z\Gsc$ is \emph{not} 
usually $F$-equivariant, since
$F$  acts by $\varphi$  on $\CA_{p'}$  and by  $\varphi$ composed with
the raising of elements  to the $q$-th
power  on $Z\Gsc$. However the  only element of the  image of $\iota$ which
occurs  in condition $(\flat)$ is  of order  2, thus raising to the
$q$-th power is trivial on it since $q$ is odd.

\subsection{Proof of Lemma \ref{lifting'}}
The next lemma reduces the proof of Lemma \ref{lifting'} to the case of one $F$-stable
quasisimple simply connected group.
\begin{lemma} \label{orbit}Let $\bG$ be simply connected such that
$\bG=\bG_1\times\cdots\times\bG_d$, where $F$ maps $\bG_i$ to $\bG_{i+1}$, the indices
being modulo $d$. Let $w=(w_1,\ldots,w_d)\in W$ and $\tA$ be a $wF$-fixed subgroup of
$\CA$. Let $v=w_1\cdot\lexp F w_d\cdot\lexp{F^2}w_{d-1}\cdots\lexp{F^{d-1}}w_2$;
assume that Lemma \ref{lifting'} holds for the pair $(\bG_1,F^d)$, the 
projection of $\tA$ on the first component and $v$,
then Lemma \ref{lifting'} holds for $(\bG,F)$, the group $\tA$ and $w$.
\end{lemma}
\begin{proof} We have $\bG^F=\bG_1^{F^d}\times\lexp F\bG_1^{F^d}\times
\cdots\times\lexp{F^{d-1}}\bG_1^{F^d}$.
The restriction to $\bG^F$ of the projection of $\bG$ onto $\bG_1$ is an isomorphism from
$\bG^F$ to $\bG_1^{F^d}$ and induces an isomorphism  $N_\bG(\bT)^{\dot wF}\simeq
N_{\bG_1}(\bT_1)^{\dot v F^d}$, if $\bT_1$ is the first component of $\bT$,
and if $\dot v=
\dot w_1\lexp F{\dot w_d}\lexp{F^2}{\dot w_{d-1}}\cdots\lexp{F^{d-1}}{\dot w_2}$, where
$\dot w=(\dot w_1,\ldots,\dot w_d)\in N_\bG(\bT)$ is a representative of $w$.
Since $\tA$ is $wF$-fixed, it is the direct product of its components $\tA_i$.
By assumption there exists $\dot v\in N_{\bG_1}(\bT_1)$ and a lift $\tau$ of
 $\tA_1$ 
to $N_{\bG_1}(\bT_1)^{\dot vF^d}$ such that the lifts of $\tA_1$ are commuting elements 
and have property $(\flat)$. Let $\dot w_i$ be an arbitrary representative
of $w_i$ for $i=2,\ldots,d$, and let $\dot w_1$ be a representative of $w_1$ such that
$\dot v=\dot w_1\lexp F{\dot w_d}\lexp{F^2}{\dot w_{d-1}}\cdots\lexp{F^{d-1}}{\dot w_2}$;
then, for $a=(a_1,\ldots,a_d)\in \tA$, using that $\tau(a_1)$ is in $N_{\bG_1}(\bT_1)^{\dot
vF^d}$,  defining inductively $\tau(a_i):=\lexp{\dot w_iF}\tau(a_{i-1})$ for $i=2,\ldots d$ gives a
lifting $\tau(a)\in N_\bG(\bT)^{\dot w F}$ satisfying $(\flat)$.
By construction $\tau(\tA)$ is a set of commuting elements.
\end{proof}

We will now look at the case of quasi-simple simply connected groups. Note
that we can choose the torus $\bT$ which appears in the beginning of
subsection \ref{case semisimple} such that in the decomposition
$\bG=\prod_i \bH_i$, the automorphism induced by $\varphi$ on the Weyl group
of the torus $\bT_i$ of $\bH_i$
(and then by $\varphi^d$ on the Weyl group of $\bG_1$ of \ref{orbit}) is trivial for
components of type $A$, $\lexp2A$, $B$, $C$, $D$,
$\lexp2D_{2r+1}$, $E_6$, $\lexp 2E_6$ or $E_7$
(we do not have to consider the cases $F_4$, $\lexp 2F_4$, $G_2$, $\lexp 2G_2$ or $E_8$ since for them
$\CA$ is trivial). On type
$D_{2n}$ we will have to consider tori where $\varphi$ induces the automorphisms
$\lexp 2D_{2n}$ or $\lexp 3D_4$.
For type $\lexp 2B_2$ we will see that $\CA^{w\varphi}$ is always trivial.

We will show that for $\bG$ simply connected quasi-simple, and $\bT\subset\bG$
a maximal torus such that $\varphi$ induced by $F$ on
$W=N_\bG(\bT)/\bT$ is such that $(W,\varphi)$ is of one of the types listed
above, we have
\begin{proposition} \label{lifting''}For any $w\in W$ and any subgroup 
$\tA\subset\CA\cap
C_W(w\varphi)$, there is a representative $\dot w$ of $w$ and a lift
 $\tau(\tA)$ of 
$\tA$ to a set of commuting elements of $N_\bG(\bT)^{\dot wF}$
such  that for each  $a\in\tA$ the element $\tau(a)$ has property $(\flat)$.
\end{proposition}
This proposition will be proved type by type in sections \ref{type AF} to \ref{type 2D2r}.

In every case but when $\bG$ is of type $D_{2n}$, the group $\CA$, thus
$\tA$ also, is cyclic. 

Let us explain how we will proceed in the case where $\tA$ is cyclic
with generator $c$ and $\varphi$ is trivial.
We need to lift $c$ to some $\dot wF$-fixed $\tau(c)$.
We will always use as lifts elements of the form $\ts(b)$ for some $b\in B(W)$,
so they will be $F$-fixed, thus the problem will be to find a representative
$\dot w$ commuting to $\ts(b)$.
We then define $\tau(c^i)$ as in Lemma \ref{c^i}, so that $\dot w$ commutes with
$\tau(c^i)$.
\begin{definition}Let $B(W)\xrightarrow\bpi W$ be the natural projection.
We say that $b\in B(W)$ has property (C) if the restriction
$C_{B(W)}(b)\xrightarrow\bpi C_W(\pi(b))$ is surjective.
\end{definition}
Thus if $b$ has property (C), and $\bpi(b)$ commutes with $w$, there exists some representative
$\bw\in B(W)$ of $w$ which commutes with $b$, thus $\dot w:=\ts(\bw)$ commutes
with $\ts(b)$ which solves our problem.

It will be the case sometimes that the $\tau(c)$ of Section \ref{section1}
is of the form $\ts(b)$ with $b$ having property (C), but often not.
We then use the following:
\begin{theorem}\label{condition C}
Let  $w\in W$ be of minimal length in its conjugacy class in $W$.
Then $\bw=\tpi(w)$ has property (C), that is the natural morphism
$C_{B(W)}(\bw)\to C_W(w)$ is surjective .
\end{theorem}
\begin{proof}
Let  $I\subset S$ be the support of $w$, that is the set of elements of $S$
which  appear in a (any) reduced  expression for $w$. Then by \cite[3.1.12,
(iii)]{GP}  the  element  $w$  is  cuspidal  in  $W_I$  (that is, no proper
parabolic  subgroup of $W_I$ contains $w$). It follows
that  $W_I$ is  the unique  minimal parabolic  subgroup containing  $w$ (if
another  parabolic subgroup $P$  with $P\not\supset W_I$  contains $w$ then
the  parabolic subgroup  $P\cap W_I$  would contain  $w$, contradicting the
cuspidality). It follows that $C_W(w)\subset N_W(W_I)$.

Now  $N_W(W_I)=W_I\rtimes N_I$ where $N_I=\{n\in W\mid \lexp nI=I \text{ and $n$
is $I$-reduced}\}$ (see for example \cite[6.1.7]{DM}).

We  can thus describe $C_W(w)$  as follows. Let $N_I^0$  be the subgroup of
$N_I$  formed  of  the  $n\in  N_I$  such  that  $\lexp  nw$  and  $w$  are
$W_I$-conjugate.  If  $n\notin  N_I^0$  then  $nW_I$ contains no element of
$C_W(w)$. Let $n\in N_I^0$; then $\lexp nw$ is equal to some $\lexp pw$ for
some   $p\in  W_I$.  The   intersection  $C_W(w)\cap  nW_I$   is  equal  to
$C_{W_I}(w)p\inv  n$.

Since $\lexp nI=I$, we have $l(\lexp nw)=l(w)$ thus $\lexp pw$ is also
of minimal length in its class in $W_I$.
Using that images by $\tpi$ of elements of
minimal  length in a conjugacy class are  conjugate in the braid group (see
\cite[3.2.7(P2)]{GP}), for any $n\in N_I^0$
we  can  find  $\bp\in B(W_I)$ such that $\tpi(\lexp nw)=\lexp\bp\bw$.

We also need \cite[Theorem 4.2]{HN} which shows that Theorem \ref{condition
 C} holds for cuspidal elements.  Applying this to  $W_I$ we see
that  for  an  element  $c\in  C_{W_I}(w)$  there  is  an  element  $\bc\in
C_{B(W_I)}(\bw)$ such that $\bpi(\bc)=c$.

Now  it is  clear that  every element  of $C_W(w)$  lifts to  an element of
$C_ {B(W)}(\bw)$. Such an  element is  with the  above notation  of the form $c
p\inv  n$  with  $c\in  C_{W_I}(w)$,  $p\in  W_I$ and $n\in N_I^0$.
We use $\bc$ and $\bp$ as above;  by the
 lemma below $\bn=\tpi(n)$ satisfies
$\lexp\bn\bw=\tpi(\lexp  nw)$. It follows  that $\bpi(\bc\bp\inv\bn)= cp\inv
n$ and $\bc\bp\inv\bn\in C_{B(W)}(\bw)$.
\begin{lemma} \label{conjI}
Let $I\subset S$ and let $n\in W$ be an $I$-reduced element such that
 $\lexp nI=I$. Then $\lexp\bn\bI=\bI$, where $\bn=\tpi(n)$.
\end{lemma}
\begin{proof}
The  assumption  is  that  for  $i\in  I$  there  exists $j\in I$ such that
$in=nj$.  Since $n$  is $I$-reduced  the members  of this  equality are two
reduced  expressions  of  the  same  element.  Thus, if $\bi=\tpi(i)$
and $\bj=\tpi(j)$,   the  equality lifts to
$\bi\bn=\bn\bj$, whence the lemma.
\end{proof}
\end{proof}
We use this as follows: 
\begin{lemma} \label{conj back}
Let $c_0$ be a minimal-length conjugate of $c\in\CA$;
let $v\in W$ be such that $c_0=v c v\inv$; let $\bv\in B(W)$ be any lift of $v$; then
\begin{enumerate}
\item $\ts(\bv\inv\bc_0\bv)$ satisfies $(\flat)$
if and only if $\ts(\bc_0)$ satisfies
$$\ts(\bc_0)^{o(c)}=\begin{cases}\iota(c^{o(c)/2})&\text{if $o(c)$ is even},\\
\iota(c^{o(c)})=1&\text{otherwise}.\end{cases} \eqno(\flat')$$ 
\item $\bv\inv\bc_0\bv$ satisfies condition (C).
\end{enumerate}
\end{lemma}
\begin{proof}
\begin{enumerate}
\item
We have
$\ts(\bv\inv\bc_0\bv)^{o(c)}=\ts(\bv)\inv\ts(\bc_0)^{o(c)}\ts(\bv)$; since the image of $\iota$ is central
  this shows that $(\flat)$ for $\ts(\bv\inv\bc_0\bv)$ is equivalent to
  $(\flat')$ for $\ts(\bc_0)$.
\item
 is obvious, since condition (C) is invariant by conjugation in $B(W)$.
\end{enumerate}
\end{proof}
Thus if we can find an element $c_0$ of minimal length such that $\bc_0$ satisfies
$(\flat')$, we are done choosing $\tau(c)=\ts(\bv\inv\bc_0\bv)$ and $\dot w=\ts(\bw)$
where $\bw$ commutes with $\bv\inv\bc_0\bv$. 
This is the canvas of the following proofs of Proposition \ref{lifting''}; we will
have to do something sightly more complicated in type $D_{2n}$.

In the following, for types $\lexp2A_n$, $\lexp2D_{2r+1}$ and $\lexp 2E_6$,
and for the untwisted types, we choose $\bT$ such that $\varphi$ acts trivially
on $W$.
\subsubsection{Type $A_n$ (and $\lexp 2A_n$)}\label{type AF}
The group $\CA$ is cyclic generated by a Coxeter element $c$ thus the subgroup $\tA$
of Proposition \ref{lifting''} is generated
by some power $c^i$ for $i|n+1$, and in \ref{type A} we
have taken $\tau(c^i)=\sigma(c)^i$ or $\tau(c^i)=-\sigma(c)^i$.
In the first case we have $\tau(c^i)=\ts(\bc^i)$. In the second case we have
$\tau(c^i)=\ts(\bc^{i+n})$ since in this case $-1=\ts(\bw_0^2)=\ts(\bc^{n+1})$.
Let us show that $\bc^i$ has property (C).
The element $\bc^i$ is periodic according to \cite[Definition 12.1]{bessis}
thus satisfies (C) by \cite[Theorem 12.4 (iii)]{bessis}. We have
checked in \ref{type A} that $\tau(c^i)$ satisfies $(\flat)$.
\subsubsection{Type $B_n$}
The group $\CA$ is of  order 2 generated by $c=w_0w_I$ with $I:=\{1,\ldots,n-1\}$. Thus the
only non trivial $\tA$ is $\CA$. In
\ref{type B} we have taken $\tau(c)=\ts(\bc)$. The centraliser of $c$
is the centraliser of $w_I$ since $w_0$ is central, hence normalises the
$1$-eigenspace of $w_I$, thus normalises the fixator of this eigenspace,
which is $W_I$. By the results of
\cite{howlett} the normaliser of $W_I$ is generated by some products of
elements of the form $w_Iw_J$ where $J=I\cup\{s\}$ for $s\in S-I$.
There is only one such element, $w_0w_I$. Hence
$C_W(c)=N_W(W_I)=\langle W_I,c\rangle$. The element $\tpi(c)=\tpi(w_0)\tpi(w_I)\inv$ still
centralises $\tpi(W_I)$; indeed $\tpi(w_0)$ is central in $B(W)$ since $w_0$ is
central in $W$ and $l(w_0s)=l(w_0)-1$ for $s\in S$, thus $\tpi(w_0)$ commutes
to $\bS$; and similarly $w_I$ lifts to a central element of $B(W_I)$.
This shows that $\bc$ satisfies property (C).
That $\ts(\bc)$ satisfies $(\flat)$ has been proved in \ref{type B}.
\subsubsection{Type $\lexp 2B_2$}
The group $\CA$ is of order 2 generated by $c=w_0w_{\{1\}}$. We have $\varphi(c)=w_0w_{\{2\}}$,
Since $c$ and $\varphi(c)$ are not conjugate, $\CA^{w\varphi}$ is always trivial
and there is nothing to do.
\subsubsection{Type $C_n$} The group $\CA$ is of order 2 generated by $c=w_0w_I$ with $I=\{2,\ldots,n\}$
and the only non trivial $\tA$ is thus $\CA$. The element
$c_0=s_ns_{n-2}s_{n-4}\cdots$  is conjugate to $c$;  this can be seen using
the  description of the conjugacy classes in $W$ by pairs of partitions (see \cite[Section
2]{stembridge}): in
the hyperoctahedral group, $w_0$ is the product of signed transpositions
$(1,-1)(2,-2)\ldots(n,-n)$  and $w_I$  is $(1,n)(2,n-1)\ldots$,
thus   the   product   is   $(1,-n)(2,-(n-1))\ldots$   of   cycle  type
$(2^{n/2},\emptyset)$  for $n$  even and  $(2^{(n-1)/2},1)$ for  $n$ odd. The
element $c_0$ has clearly the same cycle type. The
element  $c_0$  being  a  Coxeter  element  of  a standard parabolic subgroup is a
shortest  element in its  conjugacy class, see  \cite[Proposition 3.1.6 and
Lemma 3.1.14]{GP}; hence $\tpi(c_0)$ satisfies Property (C).
By Corollary \ref{sigma(involution)} we have
$\ts(\tpi(c_0))^2=(\alpha_n^\vee+\alpha_{n-2}^\vee+\cdots)/2$ which is equal to the
coweight $\varpi_1^\vee$ modulo $Y(\Tsc)$ which is $\iota(c)$ since it is the unique
non trivial element in $(Y(\Tad)/Y(\Tsc))$; hence $\ts(\tpi(c_0))$ satisfies $(\flat')$.

\subsubsection{Type $D_{2r+1}$ (and $\lexp 2D_{2r+1}$)}
The group $\CA$ is cyclic of order 4 generated by $c=w_0w_I$ with $I=\{2,\ldots,2r+1\}$. There
are two non trivial possible subgroups $\tA$, generated respectively by $c$ and $c^2$.
The shortest element in the conjugacy class of $c$ is
the element $c_0=s_1s_2s_3 s_5 s_7 s_9\cdots s_{2r+1}$:
it is conjugate to $c$ as can be seen using the
description of conjugacy classes in $W$ by pairs of partitions:
the element $w_0$ seen in the hyperoctahedral group of type $B_{2r+1}$ is the
product of signed transpositions
$(2,-2)(3,-3)\ldots$ and $w_I$ is $(1,2r+1)(2,2r)\ldots(r,r+2)$, thus
$w_0w_I$ is $(1,-(2r+1),-1,2r+1)(2,-2r)(3,-(2r-1))\ldots (r,-(r+2))(r+1,-(r+1))$ whose cycle
type is $(2^{r-1},(2,1))$, the negative cycles being $(1,-(2r+1),-1,2r+1)$ and $(r+1,-(r+1))$.
The element $c_0$ has clearly the same cycle type.
Being a Coxeter element of a standard parabolic subgroup, 
$c_0$ is a shortest element in its class, thus
 $\tpi(c_0)$ satisfies condition (C).  Now $\tpi(c_0)^4=
(\bs_1\bs_2\bs_3)^4 \bs_5^4\bs_7^4\cdots$, hence by \ref{adams-vogan}
$\ts(\tpi(c_0)^4)=(\alpha_1^\vee+\alpha_2^\vee)/2$ hence $\ts(\tpi(c_0))$
satisfies $(\flat')$ since
$(\alpha_1^\vee+\alpha_2^\vee)/2=\varpi_n^\vee=2\varpi_1^\vee=\iota(c^2)$.

We now consider the case when $\tA$ is  generated by $c^2$. The element $c^2$ is conjugate
to $c_0^2=(s_1s_2s_3)^2$ which is itself  conjugate to $c'_0:=s_1s_2$. Being a Coxeter
element of a standard parabolic subgroup, the element $c'_0$ is a shortest element in
its class, hence $c^2$ satisfies Condition (C). Since
$\ts(\tpi(s_1s_2))^2=(\alpha_1^\vee+\alpha_2^\vee)/2$, the element
$\ts(\tpi(c_0'))$ satisfies $(\flat')$.

\subsubsection{Type $E_6$ (and $\lexp 2E_6$)}
The only non-trivial subgroup $\tA$ of $\CA$ is $\CA$ itself, generated by
$c=w_0w_I$ where $I=\{1,\ldots,5\}$.
We check using \cite{chevie} that a shortest element in the conjugacy class of 
$c$ is $c_0=s_1s_3s_5s_6$:
\begin{verbatim}
julia> W=coxgroup(:E,6);c=longest(W)/longest(W,1:5);

julia> classinfo(W).classwords[position_class(W,c)]
4-element Vector{Int64}:
 1
 3
 5
 6
\end{verbatim}
We have $\tpi(c_0)^3=\tpi(w_J)^2$ where $J=\{1,3,5,6\}$
and $\ts(\tpi(w_J)^2)$ is trivial since it is thus the principal
involution in type $A_2\times A_2$, which is trivial. This shows $(\flat')$.
\subsubsection{Type $E_7$}
The only non-trivial subgroup $\tA$ of $\CA$ is $\CA$ itself, generated by
$c=w_0w_I$ where $I=\{1,\ldots,6\}$. We check using \cite{chevie} that a shortest
element in the conjugacy class of $c$ is $c_0=s_2s_5s_7$:
\begin{verbatim}
julia> W=coxgroup(:E,7);c=longest(W)/longest(W,1:6);

julia> classinfo(W).classwords[position_class(W,c)]
3-element Vector{Int64}:
 7
 5
 2
\end{verbatim}
We have $\tpi(c_0)^2=\tpi(w_J)^2$ where $J=\{2,5,7\}$, thus
$\ts(\tpi(w_J)^2)$ is the principal involution in type 
$A_1\times A_1\times A_1$, thus equal to
$(\alpha_2^\vee+\alpha_5^\vee+\alpha_7^\vee)/2$ which is $\iota(c)$ by
\ref{E7}.

\subsubsection{Type $D_{2r}$}
We first consider the case when $\varphi$ is trivial.
As in section \ref{D even}, let $a=w_0w_{\{2,3,\ldots, 2r\}}$ and
$b=w_0w_{\{1,3,\ldots, 2r\}}$. We show
that we can conjugate simultaneously $a$ and $b$ in $W$ to shortest elements in their
respective conjugacy class. Let $I=\{1,4,6,\ldots 2r\}$ and $J=\{2,4,\ldots,2r\}$.
Seen as products of signed transpositions in the hyperoctahedral group of type $B_{2r}$ we have
$$\begin{aligned}
a&=(1,-2r)(2,-(2r-1))(3,-(2r-2))\cdots(r,-(r+1))\text{ and}\\
b&=(1,\phantom{-}2r)(2,-(2r-1))(3,-(2r-2))\cdots(r,-(r+1)).\end{aligned}$$
If $r$ is even, the element $(2,-2r)(4,-(2r-2))\cdots(r,-(r+2))$ is in $W$ since it has an even
number of sign changes. This element
conjugates $a$ to $(1,2)(3,4)(5,6)\cdots(2r-1,2r)=w_J$ and $b$ to
$(1,-2)(3,4)(5,6)\cdots(2r-1,2r)=w_I$.
If $r$ is odd, the element $(2,2r)(4,-(2r-2))\cdots(r+1,-(r+1))(2r-1,-(2r-1))$ is in $W$ and
conjugates $a$ to $w_I$ and $b$ to $w_J$. 
The elements $w_I$ and $w_J$ being Coxeter elements of standard parabolic subgroups are
shortest elements in their conjugacy classes.
If $r$ is odd we let $a_0=w_I$ and $b_0=w_J$, and if $r$ is even we
let $a_0=w_J$ and $b_0=w_I$, so that in any case $a_0$ is an element of
minimal length conjugate to $a$ and $b_0$ an element of minimal length
conjugate to $b$. Since this is done by a simultaneous conjugation, it
follows also that $c_0=s_1 s_2$ is an element (of minimal length) conjugate
to $c=ab$.

We have to show that for any subgroup $\tA$ of $\CA$, we can lift
$\tA$ to a $\tau(\tA)$ satisfying $(\flat)$ which is the image
by $\ts$ of a set of elements ``satisfying property (C)'' (or an analogous condition
when $\tA=\CA$, see below).

We first deal with cyclic subgroups $\tA$. As in previous cases, we replace $a,b,c$
by elements of minimal length in their class $a_0,b_0,c_0$. We let
$\ba_0=\tpi(a_0)$, $\bb_0=\tpi(b_0)$ and $\bc_0=\tpi(c_0)$.
Using Lemma \ref{conj back} we check that $\ts(\ba_0)$,$\ts(\bb_0)$
and $\ts(\bc_0)$ satisfy property $(\flat')$: indeed by Proposition \ref{adams-vogan}
we have $$\ts(\ba_0)^2=\begin{cases}
1/2\sum_{i\in I}\alpha_i^\vee \text{ if $r$ is odd}\\
 1/2\sum_{i\in J}\alpha_i^\vee \text{ if $r$ is even}
\end{cases}\equiv\varpi_1^\vee\pmod {Y(\bT)},$$
 and 
$$\ts(\bb_0)^2=\begin{cases}
1/2\sum_{i\in J}\alpha_i^\vee \text{ if $r$ is odd}\\
 1/2\sum_{i\in I}\alpha_i^\vee \text{ if $r$ is even}
\end{cases}\equiv\varpi_2^\vee\pmod{Y(\bT)},$$
and $\ts(\bc_0)^2=(\alpha_1^\vee+\alpha_2^\vee)/2\equiv \varpi_n^\vee\pmod{Y(\bT)}$.

It remains the case when $\tA=\CA$.
We will show that the subgroup of $B(W)$ generated by $\ba_0$ and $\bb_0$
satisfies the analogue of property (C), that is the natural map 
$C_{B(W)}(\ba_0)\cap C_{B(W)}(\bb_0)\to C_W(a_0)\cap C_W(b_0)$ is surjective.
Let us first compute $C_W(w_J)$. According to the proof of Theorem
\ref{condition C}, we have $C_W(w_J)\subset W_J\rtimes N_J$. Since $W_J$ is
commutative and $w_J$ is stable by any diagram automorphism, we have actually
$C_W(w_J)=W_J\rtimes N_J$. According to \cite{howlett}, $N_J$ is generated
by the elements $v_i:=w_{J\cup\{i\}}w_J$ for all letters $i\notin J$.
These are $v_1=s_1$ and $v_3=w_{\{2,3,4\}}w_{\{2,4\}}$,
$v_5=w_{\{4,5,6\}}w_{\{4,6\}}$
\etc Thus $C_W(w_J)$ is generated by $s_1,s_2,s_4,s_6,\ldots,s_{2r}$,
which commute to each other, and $v_3$, which swaps $s_2$ and $s_4$, $v_5$
\etc A computation in $\fS_6$ shows that $v_3$ and $v_5$
satisfy the braid relation $v_3 v_5 v_3=v_5 v_3 v_5$, and similarly for
$v_5$ and $v_7$, \etc It follows that the group generated by 
$v_3, v_5,\ldots,v_{2r-1}$ is isomorphic to $\fS_r$. Let us compute
the intersection of $C_W(w_J)$ with $C_W(w_I)$. It
is the same as the intersection with $C_W(s_1 s_2)$ since $w_I=s_1 s_2 w_J$. All generators of
$C_W(w_J)$ commute with $s_1 s_2$ excepted $v_3$. In $\fS_r$ all elements
have an expression in the generators which involves only one time the first
generator. This results from the fact that $\fS_r$ has only two double
cosets with respect to $\fS_{r-1}$ acting on the second to last letter, the
trivial one and that of the first generator. It follows that in $C_W(W_J)$
each element has an expression involving $v_3$ only once. This implies
that $C_W(W_J)\cap C_W(s_1s_2)$ is generated by all generators of $C_W(W_J)$
but $v_3$, that is $s_1, s_2, s_4,\ldots s_{2r}, v_5,\ldots v_{2r-1}$.
These elements all lift to $B(W)$ by $\tpi$ to elements centralising
$\ba_0$ and $\bb_0$ (invoking Lemma \ref{conjI} for the $v_j$).

It follows that if we let $v$ be the element such that $a_0=vav\inv$,
$b_0=vbv\inv$, and let $\bv$ be a lift of $v$, the elements
$\ts(\bv\inv\ba_0\bv)$ and $\ts(\bv\inv\bb_0\bv)$  lift $a$ and $b$,
satisfy $(\flat)$ by Lemma \ref{conj back} and the elements 
$\bv\inv\ba_0\bv$ and $\bv\inv\bb_0\bv$ satisfy the analogue of (C).
We lift $c$ by  $\ts(\bv\inv\ba_0\bb_0\bv)$.

In the next two cases we choose $\bT$ to be quasi-split in $\bG$.
\subsubsection{Type $\lexp 3D_4$}
We claim there is nothing to do since for any $w\in W$ the only
$w\varphi$-fixed element of $\CA$ is the identity. Indeed, $a$, $b$ and $c$
belong to different conjugacy classes of $W$ which are permuted circularly by
$\varphi$.

\subsubsection{Type $\lexp 2D_{2r}$}\label{type 2D2r}

The only element of $\CA$ which may be $w\varphi$-fixed is $c$, since
$a$ and $b$ belong to different conjugacy classes of $W$ exchanged by $\varphi$.
Since $c$ is fixed by $\varphi$, we have to lift $w\varphi$ for $w\in C_W(c)$.
We proceed as in the case $\varphi$ trivial by conjugating $c$ to a minimal
length element in its class. Let $c_0=s_1s_2=\lexp vc$ be such an element.
Let $w'\varphi=\lexp v(w\varphi)$. Since $\bc_0=\tpi(c_0)$ has property (C)
and $c_0$ and $\bc_0$ are $\varphi$-fixed, there exists a lift $\bw'$ of $w'$
such that $\bw'\varphi$ commutes to $\bc_0$.
For any lift $\bv\in B(W)$ of $v$, the conjugate
$\bv\inv\bw'\varphi\bv$ lifts $w\varphi$ and centralises
$\bv\inv\bc_0\bv$, a lift of $c$. Applying $\ts$, we get a lift $\dot wF$ of 
$w\varphi$ centralising a lift to $N_\bG(\bT)$ of $c$.
\bibliography{centraliser}
\bibliographystyle{plain}
\end{document}